\documentclass[leqno, 12pt]{article}
\usepackage{amsmath,amsfonts,amsthm,amssymb,indentfirst,mathrsfs}
\usepackage{xy} \xyoption{all}

\usepackage[colorlinks]{hyperref}
\hypersetup{linkcolor=blue, urlcolor=blue, citecolor=red}

\newtheorem{theorem}{Theorem}
\newtheorem{lemma}[theorem]{Lemma}
\newtheorem{corollary}[theorem]{Corollary}
\newtheorem{proposition}[theorem]{Proposition}

\theoremstyle{definition}
\newtheorem{example}[theorem]{Example}

\newtheorem{definition}[theorem]{Definition}

\newcommand{\Aut}{\mathrm{Aut}}
\newcommand{\pth}{\mathrm{Path}}

\newcommand{\Z}{\mathbb{Z}}

\newcommand{\N}{\mathbb{N}}
\newcommand{\so}{\mathbf{s}}
\newcommand{\ra}{\mathbf{r}}
\newcommand{\GL}{\mathscr{L}}
\newcommand{\GR}{\mathscr{R}}
\newcommand{\GJ}{\mathscr{J}}
\newcommand{\GH}{\mathscr{H}}
\newcommand{\GD}{\mathscr{D}}
\newcommand{\GP}{\mathscr{P}}

\begin{document}

\title{The Structure of a Graph Inverse Semigroup}
\author{Zachary Mesyan and J.\ D.\ Mitchell}

\maketitle

\begin{abstract}

Given any directed graph $E$ one can construct a \emph{graph inverse semigroup} $G(E)$, where, roughly speaking, elements correspond to paths in the graph. In this paper we study the semigroup-theoretic structure of $G(E)$. Specifically, we describe the non-Rees congruences on $G(E)$, show that the quotient of $G(E)$ by any Rees congruence is another graph inverse semigroup, and classify the $G(E)$ that have only Rees congruences. We also find the minimum possible degree of a faithful representation by partial transformations of any countable $G(E)$, and we show that a homomorphism of directed graphs can be extended to a homomorphism (that preserves zero) of the corresponding graph inverse semigroups  if and only if it is injective.

\medskip


\end{abstract}

\section{Introduction}

A \emph{graph inverse semigroup} $G(E)$ is a semigroup constructed from a
directed graph $E$ (to be defined precisely below), where, roughly speaking,
elements correspond to paths in the graph. These semigroups were introduced by Ash/Hall~\cite{AH} in order to show that every partial order can be realized as that of the nonzero $\GJ$-classes of an inverse semigroup. Graph inverse semigroups also generalize \emph{polycyclic monoids,} first defined by Nivat/Perrot~\cite{NP}, and arise in the study of rings and $C^*$-algebras. More specifically, for any field $K$ and any directed graph $E$, the (contracted) semigroup ring $KG(E)$ is called the \emph{Cohn path $K$-algebra} of $E$, and the quotient of a Cohn path algebra by a certain ideal is known as the \emph{Leavitt path $K$-algebra} of $E$. These rings were introduced independently by Abrams/Aranda Pino~\cite{AP} and
Ara/Moreno/Pardo~\cite{AMP}. Cohn path algebras and Leavitt path algebras are
algebraic analogues of Toeplitz $C^*$-algebras and graph $C^*$-algebras (see~\cite{KPRR, KPR}), respectively. The connection of graph inverse semigroups to rings is discussed in more detail in~\cite{MMMP}, while their connection to $C^*$-algebras is covered in~\cite{Paterson}. There is extensive literature devoted to all of the algebras mentioned above. Graph inverse semigroups also have been studied in their own right in recent years~\cite{CS, Jones, JL, Krieger, MMMP}. 

The goal of the present paper is to describe the semigroup-theoretic structure
of an arbitrary graph inverse semigroup $G(E)$, with particular emphasis on
the relationship between properties of semigroups and properties of graphs. After recalling some known facts about the ideals of $G(E)$ and describing the partially ordered set of its $\GJ$-classes (Proposition~\ref{partial-orders}), we study in detail the congruences on graph inverse semigroups and their corresponding quotients. Specifically, we show that the quotient of any $G(E)$ by a Rees congruence is always isomorphic to another graph inverse semigroup (Theorem~\ref{Rees-quotients}), describe the non-Rees congruences on these semigroups (Proposition~\ref{Rees-char}), and completely classify those $G(E)$ that have only Rees congruences, in terms of properties of $E$ (Theorem~\ref{only-Rees}). Then we find the minimum possible degree of a faithful representation by partial transformations of an arbitrary countable graph inverse semigroup (Proposition~\ref{repres}). In particular, for finite $G(E)$ this degree is the number of paths in $E$ ending in vertices with out-degree at most $1$. We also show that a homomorphism of directed graphs can be extended to a homomorphism of the corresponding graph inverse semigroups (that preserves zero) if and only if it is injective (Theorem~\ref{graph-iso}). From this we conclude that the automorphism group of any graph $E$ is isomorphic to the automorphism group of the corresponding semigroup $G(E)$ (Corollary~\ref{cor-autos}) and that every group can be realized as the automorphism group of some graph inverse semigroup (Corollary~\ref{group-cor}). The relevant concepts from semigroup theory and graph theory are reviewed in the next section.

Some of the results in this paper were suggested by computations obtained using the \emph{Semigroups} GAP package~\cite{GAP}.

\subsection*{Acknowledgements} 
We would like to thank Benjamin Steinberg for pointing us to relevant literature. We are also grateful to the referee for helpful suggestions about improving the paper and for noticing a gap in the proof of an earlier version of Lemma~\ref{non-Rees}.

\section{Definitions}\label{def-sect}

\subsection{Semigroups} \label{semigp-sect}
We begin by recalling some standard notions from semigroup theory. The readers familiar with the field may wish to skip this subsection, and refer to it as necessary.

Let $S$ be a semigroup. Then $S$ is an \emph{inverse semigroup} if for each $x \in S$ there is a unique element $x^{-1} \in S$ satisfying $x = xx^{-1}x$ and $x^{-1} = x^{-1}xx^{-1}$. By $S^1$ we shall mean the monoid obtained from $S$ by adjoining an identity element (if $S$ does not already have such an element). The following relations on elements $x,y \in S$ are known as \emph{Green's relations:}\\
$(1)$ $x \, \GL \, y$ if and only if $S^1 x = S^1 y$,\\
$(2)$ $x \, \GR \, y$ if and only if $x S^1 = y S^1$,\\
$(3)$ $x \, \GJ \, y$ if and only if $S^1 x S^1 = S^1 y S^1$,\\
$(4)$ $x \, \GH \, y$ if and only if $x \, \GL \, y$ and $x \, \GR \, y$,\\
$(5)$ $x \, \GD \, y$ if and only if $x \, \GL \, z$ and $z \, \GR \, y$ for some $z \in S$.\\
Each of these is an equivalence relation, and we denote by $L_x$, $R_x$, and $J_x$ the $\GL$-class, $\GR$-class, and $\GJ$-class of $x$, respectively. The following define partial orders on these classes:
$(1)$ $L_x\leq_{\GL}L_y$ if and only if $S^1 x \subseteq S^1 y$,\\
$(2)$ $R_x\leq_{\GR}R_y$ if and only if $xS^1 \subseteq yS^1$,\\
$(3)$ $J_x\leq_{\GJ}J_y$ if and only if $S^1xS^1 \subseteq S^1yS^1$.

We denote by $\N$ and $\Z$ the semigroups of the natural numbers and the integers, respectively, under addition.

\subsection{Graphs}\label{graphs-sect}

A \emph{directed graph} $E=(E^0,E^1,\ra,\so)$ consists of two sets $E^0,E^1$ (containing \emph{vertices} and \emph{edges}, respectively), together with functions $\so,\ra:E^1 \to E^0$, called \emph{source} and \emph{range}, respectively. A \emph{path} $x$ in $E$ is a finite sequence of (not necessarily distinct) edges $x=e_1\dots e_n$ such that $\ra(e_i)=\so(e_{i+1})$ for $i=1,\dots,n-1$. In this case, $\so(x):=\so(e_1)$ is the \emph{source} of $x$, $\ra(x):=\ra(e_n)$ is the \emph{range} of $x$, and $|x|:=n$ is the \emph{length} of $x$. If $x = e_1\dots e_n$ is a path in $E$ such that $\so(x)=\ra(x)$ and $\so(e_i)\neq \so(e_j)$ for every $i\neq j$, then $x$ is called a \emph{cycle}. A cycle consisting of one edge is called a \emph{loop}. The graph $E$ is \emph{acyclic} if it has no cycles. We view the elements of $E^0$ as paths of length $0$ (extending $\so$ and $\ra$ to $E^0$ via $\so(v)=v$ and $\ra(v)=v$ for all $v\in E^0$), and denote by $\pth(E)$ the set of all paths in $E$. Given a vertex $v \in E^0$, $|\{e \in E^1 \mid \so(e) = v\}|$ is called the \emph{out-degree} of $v$, while $|\{e \in E^1 \mid \ra(e) = v\}|$ is the \emph{in-degree} of $v$. (If $X$ is any set, then $|X|$ denotes the cardinality of $X$.) A vertex $v \in E^0$ is a \emph{sink} if it has out-degree $0$. A \emph{strongly connected component} of $E$ is a directed subgraph $F$ maximal with respect to the property that for all $v,w \in F^0$ there is some $p \in \pth(F)$ such that $\so(p)=v$ and $\ra(p)=w$.

We say that a directed graph $E$ is \emph{simple} if it has no loops, and for all distinct $v, w \in E^0$ there is at most one $e \in E^1$ such that $\so(e)=v$ and $\ra(e)=w$. A directed graph $E$ is \emph{finite} if $E^0$ and $E^1$ are both finite. From now on we shall refer to directed graphs as simply ``graphs".

Let $E_a=(E_a^0,E_a^1,\ra_a,\so_a)$ and $E_b=(E_b^0,E_b^1,\ra_b,\so_b)$ be two graphs, and let $\phi_0 : E_a^0 \to E_b^0$ and $\phi_1 : E_a^1 \to E_b^1$ be functions. Then the pair $\phi = (\phi_0, \phi_1)$ is a \emph{graph homomorphism from} $E_a$ \emph{to} $E_b$ if $\phi_0(\so_a(e)) = \so_b(\phi_1(e))$ and $\phi_0(\ra_a(e)) = \ra_b(\phi_1(e))$ for every $e \in E_a^1$. If $\phi_0$ and $\phi_1$ are in addition bijective, then $\phi$ is a \emph{graph isomorphism from} $E_a$ \emph{to} $E_b$. In this case we say that $E_a$ and $E_b$ are \emph{isomorphic} and write $E_a \cong E_b$.

\subsection{Graph Inverse Semigroups} \label{gis-sect}

Given a graph $E=(E^0,E^1,\ra,\so)$, the \emph{graph inverse semigroup $G(E)$ of $E$} is the semigroup with zero generated by the sets $E^0$ and $E^1$, together with a set of variables $\{e^{-1} \mid e\in E^1\}$, satisfying the following relations for all $v,w\in E^0$ and $e,f\in E^1$:\\
(V)  $vw = \delta_{v,w}v$,\\ 
(E1) $\so(e)e=e\ra(e)=e$,\\
(E2) $\ra(e)e^{-1}=e^{-1}\so(e)=e^{-1}$,\\
(CK1) $e^{-1}f=\delta _{e,f}\ra(e)$.\\
(Here $\delta$ is the Kronecker delta.) We define $v^{-1}=v$ for each $v \in E^0$, and for any path $y=e_1\dots e_n$ ($e_1,\dots, e_n \in E^1$) we let $y^{-1} = e_n^{-1} \dots e_1^{-1}$. With this notation, every nonzero element of $G(E)$ can be written uniquely as $xy^{-1}$ for some $x, y \in \pth(E)$, by the CK1 relation. It is also easy to verify that $G(E)$ is indeed an inverse semigroup, with $(xy^{-1})^{-1} = yx^{-1}$ for all $x, y \in \pth (E)$.

If $E$ is a graph having only one vertex $v$ and $n$ edges (necessarily loops), for some integer $n \geq 1$, then $G(E)$ is known as a \emph{polycyclic monoid}, and is denoted by $P_n$. We note that $P_1$, also called the \emph{bicyclic monoid}, is typically defined in the literature without a zero element.

\section{Ideals} 

The following characterizations of Green's relations and their associated equivalence classes on graph inverse semigroups will be useful throughout the paper. These characterizations are also given by Jones in~\cite{Jones}, but we include the short proofs for completeness.

\begin{lemma}\label{orders}
Let $E$ be any graph, and let $u,v,x,y \in \pth (E)$ be such that $\, \ra(u)=\ra(v)$ and $\, \ra(x)=\ra(y)$. Then the following hold.
\begin{enumerate}
\item[$(1)$] $L_{uv^{-1}}\leq_{\GL} L_{xy^{-1}}$ if and only if $v=yt$ for some $t \in \pth(E)$.
\item[$(2)$] $R_{uv^{-1}}\leq_{\GR} R_{xy^{-1}}$ if and only if $u=xt$ for some $t \in \pth(E)$.
\item[$(3)$]  $J_{uv^{-1}}\leq_{\GJ} J_{xy^{-1}}$ if and only if $\, \so(t)=\ra(x)$ and $\, \ra(t)=\ra(u)$ for some $t \in \pth(E)$.
\end{enumerate}
\end{lemma}

\begin{proof}
(1) $L_{uv^{-1}}\leq_{\GL} L_{xy^{-1}}$ if and only if $G(E)^1 uv^{-1} \subseteq G(E)^1 xy^{-1}$ if and only if $G(E) v^{-1} \subseteq G(E) y^{-1}$, since $u^{-1}, x^{-1}, \ra(v), \ra(y) \in G(E)$. The latter is equivalent to $v^{-1} \in G(E) y^{-1}$, which is in turn equivalent to $v^{-1} = t^{-1}y^{-1}$ for some $t \in \pth(E)$, that is $v=yt$.

\vspace{\baselineskip}

(2) Analogously to the proof of (1), $R_{uv^{-1}}\leq_{\GR} R_{xy^{-1}}$ if and only if $uv^{-1} G(E)^1 \subseteq xy^{-1} G(E)^1$ if and only if $u \in x G(E)$ if and only if $u=xt$ for some $t \in \pth(E)$.

\vspace{\baselineskip}

(3) Since $xy^{-1}=x\ra(x)y^{-1}$ and $\ra(x) = \ra(x)\ra(y) = x^{-1}(xy^{-1})y$, we have $G(E) xy^{-1} G(E) = G(E) \ra(x) G(E)$.

Now suppose that there is $t \in \pth(E)$ such that $\so(t)=\ra(x)$ and $\ra(t)=\ra(u)$. Then $$\ra(u) = \ra(t) = t^{-1}t = t^{-1}\so(t)t \in G(E)\ra(x)G(E).$$ Hence $uv^{-1} = u\ra(u)v^{-1} \in G(E)\ra(x)G(E)$, and since $G(E)\ra(x)G(E) = G(E) xy^{-1} G(E)$, we have $uv^{-1} \in G(E) xy^{-1} G(E)$. It follows that $J_{uv^{-1}}\leq_{\GJ} J_{xy^{-1}}$.

Conversely, if $J_{uv^{-1}}\leq_{\GJ} J_{xy^{-1}}$, then $uv^{-1} \in G(E) xy^{-1} G(E)$, and therefore $$\ra(u) = u^{-1}uv^{-1}v \in G(E) xy^{-1} G(E) = G(E) \ra(x) G(E).$$ Hence $\ra(u) = st^{-1}rp^{-1}$ for some $r,p,s,t \in \pth (E)$ with $\so(t) = \ra(x) = \so(r)$, $\ra(t)= \ra(s)$, and $\ra(r)= \ra(p)$. By the uniqueness of the representations of elements of $G(E)$ discussed in Section~\ref{gis-sect}, for $st^{-1}rp^{-1}$ to be a vertex we must have $s,p \in E^0$. Hence $s=\ra(u)=p$, and therefore $\ra(u) = t^{-1}r$. It follows that $r=t$, and in particular, $\so(t)=\ra(x)$ and $\ra(t)= \ra(s) = \ra(u)$, as required.
\end{proof}

\begin{corollary} \label{Green's}
Let $E$ be any graph, and let $u,v,x,y \in \pth (E)$ be such that $\, \ra(u)=\ra(v)$ and $\, \ra(x)=\ra(y)$. Then the following hold.
\begin{enumerate}
\item[$(1)$] $uv^{-1} \, \GL \, xy^{-1}$ if and only if $v=y$. 
\item[$(2)$] $uv^{-1} \, \GR \, xy^{-1}$ if and only if $u=x$. 
\item[$(3)$] $uv^{-1} \, \GJ \, xy^{-1}$ if and only if $\, \ra(u)$ and $\, \ra(x)$ are in the same strongly connected component of $E$.
\item[$(4)$] $uv^{-1} \, \GH \, xy^{-1}$ if and only if $uv^{-1}=xy^{-1}$.
\item[$(5)$] $uv^{-1} \, \GD \, xy^{-1}$ if and only if $\, \ra(u)=\ra(x)$.
\end{enumerate}
\end{corollary}

\begin{proof}
To prove (1) we note that $uv^{-1} \, \GL \, xy^{-1}$ if and only if $L_{uv^{-1}} = L_{xy^{-1}}$, which is equivalent to $v=y$, by Lemma~\ref{orders}(1). The proofs of (2) and (3) are analogous, while (4) follows from (1) and (2). For (5), we have $uv^{-1} \, \GD \, xy^{-1}$ if and only if $uv^{-1} \, \GL \, rp^{-1}$ and $rp^{-1} \, \GR \, xy^{-1}$ for some $r,p \in \pth(E)$ such that $\ra(r)=\ra(p)$. By (1) and (2), this is equivalent to $v=p$ and $r=x$ for some $r,p \in \pth(E)$ such that $\ra(r)=\ra(p)$, which is equivalent to $\ra(x)= \ra(v)=\ra(u)$.
\end{proof}

It follows from Corollary~\ref{Green's}(3) that there is a one-to-one correspondence between the strongly connected components of the graph $E$ and the nonzero $\mathscr{J}$-classes of $G(E)$. In particular, if $E$ acyclic, then the nonzero $\mathscr{J}$-classes are in correspondence with the vertices of $E$. 

In the next proposition we describe the structure of the partial order of nonzero $\GJ$-classes of a graph inverse semigroup. First, we note that if $E$ is a simple graph, then every edge is uniquely determined by its source and range vertices, and hence $E^1$ can be identified with the subset $\{(\so(e),\ra(e)) \mid e \in E^1\}$ of $E^0\times E^0$.

\begin{proposition} \label{partial-orders}
Let $E$ be a graph, and let $C(E)$ be the set of strongly connected components of $E$. Also let $$B(E) = \{(U,V) \mid U\neq V \text{ and } \, \so(e)\in V^0, \ra(e)\in U^0 \text{ for some } e\in E^1 \} \subseteq C(E)\times C(E),$$ and let $E_S$ be the simple graph defined by $E_S^0 = C(E)$ and $E_S^1 = B(E)$.

Then the following partially ordered sets are order-isomorphic:
\begin{enumerate}
\item[$(a)$] the set of nonzero $\GJ$-classes of $G(E)$ with the partial order $\leq_{\GJ}$,
\item[$(b)$] the set of nonzero $\GJ$-classes of $G(E_S)$ with the partial order $\leq_{\GJ}$,
\item[$(c)$] $C(E)$ with the least transitive reflexive binary relation containing $B(E)$.
\end{enumerate}
\end{proposition}

\begin{proof}
First, note that $(U,V) \in C(E)\times C(E)$ belongs to the relation given in (c) if and only if there is a path $p \in \pth(E)$ with $\so(p) \in V^0$ and $\ra(p) \in U^0$ (which includes the case where $U=V$). It follows from this that if $(U,V)$ and $(V,U)$ belong to this relation, then $U=V$, and hence that the relation is necessarily antisymmetric, making it a partial order.

By Corollary~\ref{Green's}(3), every nonzero $\GJ$-class of $G(E)$ contains a vertex, and two vertices in $E^0$ belong to the same $\GJ$-class if and only if they are in the same strongly connected component of $E$. Thus, the map $\varphi_1$ from the set defined in $(a)$ to the set defined in $(c)$, that takes each $\GJ$-class to the strongly connected component containing the vertices in that $\GJ$-class, is well-defined and bijective. Analogously, the map $\varphi_2$ from the set defined in $(b)$ to the set defined in $(c)$, that takes each $\GJ$-class to the strongly connected component containing the unique vertex in that $\GJ$-class, is well-defined and bijective.

Now, by Lemma~\ref{orders}(3), $J_u \leq_{\GJ} J_v$ if and only if there is a path in $E$ from $v$ to $u$, for all $u,v \in E^0$, and similarly for $E_S$. Hence, it follows from the definition of $B(E)$ that $\varphi_1$ and $\varphi_2$ respect the partial orders on their domains, and are therefore order-isomorphisms.
\end{proof}



As a consequence of Proposition~\ref{partial-orders} we obtain the following result of Ash and Hall.

\begin{corollary}[Theorem 4(i) in \cite{AH}] \label{poset-thrm}
Every partially ordered set is order-isomorphic to the set of nonzero $\GJ$-classes of $G(E)$ with the partial order $\, \leq_{\GJ}$, for some graph $E$.
\end{corollary}

\begin{proof}
This follows from Proposition~\ref{partial-orders}, since any partially ordered set can be obtained by taking the transitive reflexive closure of a binary relation of the form $B(E)$ in the proposition.
\end{proof}

In contrast to Corollary~\ref{poset-thrm}, the possible partial order structures on the sets of $\GL$-classes and $\GR$-classes of $G(E)$ (with the partial orders $\leq_{\GL}$ and $\leq_{\GR}$, respectively) are rather limited, as the next lemma (which follows immediately from Lemma~\ref{orders}(1,2)) shows.

\begin{lemma}\label{LR-partial}
Let $E$ be a graph and $v \in E^0$. Then $L_v$ is a maximal element with respect to $\leq_{\GL}$, and $R_v$ is a maximal element with respect to $\leq_{\GR}$.
\end{lemma} 

For example, it follows from the above lemma that up to order-isomorphism the only totally ordered set with more than one element that can be realized as the nonzero $\GR$-classes of $G(E)$ with the partial order $\leq_{\GR}$ is the set of the negative integers (with the usual ordering). For, by Lemma~\ref{LR-partial}, the nonzero $\GR$-classes of $G(E)$ are totally ordered only if $|E^0|=1$. Moreover, there can be at most one edge in $E^1$ (if $e,f \in E^1$ were distinct, then $R_e$ and $R_f$ would be incomparable, by Lemma~\ref{orders}(2)). Thus, either $E^1$ is empty, in which case $G(E)$ has exactly one nonzero $\GR$-class, or $E^1\ = \{e\}$, in which case the nonzero $\GR$-classes are related as follows: $$\dots \leq_{\GR} R_{e^3} \leq_{\GR} R_{e^2} \leq_{\GR} R_{e} \leq_{\GR} R_v.$$ By a similar argument, the same holds for the $\GL$-classes of $G(E)$.

\section{Congruences}

Recall that given a semigroup $S$, an equivalence relation $R \subseteq S \times S$ is a \emph{congruence} if  $(x, y) \in R$ implies that $(xz, yz), (zx, zy) \in R$ for all $x, y, z \in S$. The \emph{diagonal congruence} on $S$ is the relation $\Delta =  \{(x,x) \mid x \in S\}$. A congruence $R \subseteq S \times S$ is a \emph{Rees congruence} if $R = (I \times I) \cup \Delta$ for some ideal $I$ of $S$. Note that if $S$ has a zero element, then $\Delta$ is the Rees congruence corresponding to the zero ideal. Also, $S$ is \emph{congruence-free} if its only congruences are $S \times S$ and $\Delta$. 

We begin our investigation of the congruences on graph inverse semigroups by describing the quotients of these semigroups by Rees congruences.

\begin{definition}
Let $E$ be a graph and $S \subseteq E^0$. By $E\setminus S$ we shall denote the graph $F = (F^0,F^1,\ra_F,\so_F)$, where $F^0 = E^0 \setminus S$, $F^1 = E^1 \setminus \{e \in E^1 \mid \so(e) \in S \text{ or } \ra(e) \in S\}$, and $\ra_F$, $\so_F$ are the restrictions of $\ra$, $\so$, respectively, to $F^1$.
\end{definition}

\begin{theorem} \label{Rees-quotients}
Let $E$ be a graph, and let $R \subseteq G(E) \times G(E)$ be a Rees congruence. Then $G(E)/R \cong G(E\setminus (I \cap E^0))$, where $I$ is the ideal of $G(E)$ corresponding to $R$.
\end{theorem}

\begin{proof}
Write $R = (I \times I) \cup \{(\mu,\mu) \mid \mu \in G(E)\}$ where $I$ is an ideal of $G(E)$, and let $F = E\setminus (I \cap E^0)$. Define $\varphi : G(E) \to G(F)$ by 
$$\varphi (xy^{-1}) = 
\left\{ \begin{array}{ll}
xy^{-1} & \text{if } x,y \in \pth(F)\\
0 & \text{otherwise}
\end{array}\right.$$
for all $x,y \in \pth(E)$, and $\varphi(0)=0$. We note that if $uv^{-1} \in I$ for some $u,v \in \pth(E)$ with $\ra(u)=\ra(v)$, then $\ra(u) \in I$, by Corollary~\ref{Green's}(3), and hence $u,v \not\in \pth(F)$. It follows that $\varphi(\mu) = 0$ if and only if $\mu \in I$, for all $\mu \in G(E)$.

To show that $\varphi$ is a homomorphism, let $\mu, \nu \in G(E)$. If either $\mu \in I$ or $\nu \in I$, then $\mu \nu \in I$, and therefore $\varphi(\mu)\varphi(\nu)=0=\varphi(\mu \nu)$. Let us therefore suppose that $\mu, \nu \not\in I$. Then $\varphi(\mu) = \mu$ and $\varphi(\nu) = \nu$, by the definition of $\varphi$ and the previous paragraph. Thus, if $\mu \nu = 0$, then $$\varphi(\mu)\varphi(\nu) = \mu\nu = 0 = \varphi(\mu \nu).$$ Let us therefore further assume that $\mu \nu \neq 0$, and write $\mu = uv^{-1}$, $\nu = xy^{-1}$ ($u,v,x,y \in \pth(E)$). Then there is some $t \in \pth(E)$ such that either $v=xt$ or $x=vt$. In the first case $uv^{-1}xy^{-1} = ut^{-1}y^{-1}$. Since $uv^{-1}\not\in I$ and $uv^{-1}\GJ ut^{-1}y^{-1}$, by Corollary~\ref{Green's}(3), it follows that $ut^{-1}y^{-1}\not\in I$. Thus $$\varphi(\mu)\varphi(\nu) = \varphi(uv^{-1})\varphi(xy^{-1}) = uv^{-1}xy^{-1} =  ut^{-1}y^{-1} = \varphi(ut^{-1}y^{-1}) = \varphi(\mu \nu).$$ If, on the other hand, $x=vt$, then $uv^{-1}xy^{-1} = uty^{-1}$. Again, since $xy^{-1} \not\in I$ and $xy^{-1}\GJ uty^{-1}$, by Corollary~\ref{Green's}(3), it follows that $uty^{-1}\notin I$. Thus $$\varphi(\mu)\varphi(\nu) = \varphi(uv^{-1})\varphi(xy^{-1}) = uv^{-1}xy^{-1} =  uty^{-1} = \varphi(uty^{-1}) = \varphi(\mu \nu).$$ Since in every case $\varphi(\mu)\varphi(\nu)=\varphi(\mu \nu)$, we conclude that $\varphi$ is a homomorphism.

Since $\varphi(\mu) = 0$ if and only if $\mu \in I$, for all $\mu \in G(E)$, it follows that $R = \{(\mu, \nu) \in G(E) \mid \varphi(\mu) = \varphi(\nu)\}$, which, by definition, is the kernel of $\varphi$. Since $\varphi$ is clearly surjective, by the first isomorphism theorem for semigroups \cite[Theorem 1.5.2]{Howie}, $G(E)/R \cong G(F)$.
\end{proof}

Turning to non-Rees congruences, the next proposition shows how they arise.

\begin{proposition} \label{Rees-char}
Let $E$ be a graph and $R \subseteq G(E) \times G(E)$ a congruence. Then $R$ is a non-Rees congruence if and only if there exists $v \in E^0$ such that $\, (v, \mu) \in R$ for some $\mu \in G(E)\setminus \{v\}$, but $\, (v,0) \notin R$.

Moreover, if $v \in E^0$ is such that $\, (v, \mu) \in R$ for some $\mu \in G(E)\setminus \{v\}$, but $\, (v,0) \notin R$, then $v$ must satisfy the following conditions.
\begin{enumerate}
\item[$(1)$] Let $S = \{\mu \in G(E) \mid (v, \mu) \in R\}$. Then $S$ is an inverse semigroup, and every element of $S$ is of the form $xpx^{-1}$ or $xp^{-1}x^{-1}$ for some $x,p\in \pth(E)$ satisfying $\, \so(x) = v$ and $\, \ra(x)=\so(p)=\ra(p)$.
\item[$(2)$] There exists $e \in E^1$ with $\, \so(e)=v$, such that every $p \in \pth(E)\setminus E^0$ with $\, \so(p)=v$ and $\, \ra(p) = \ra(e)$ is of the form $p = et$ for some $t \in \pth(E)$.
\end{enumerate}
\end{proposition}

\begin{proof}
Suppose that for all $v \in E^0$ such that $(v, \mu) \in R$ for some $\mu \in G(E)\setminus \{v\}$, we have $(v,0) \in R$. Let $\nu \in G(E) \setminus \{0\}$ be any element such that $(\nu, \mu) \in R$ for some $\mu \in G(E)\setminus \{\nu\}$, and write $\nu = xy^{-1}$ ($x,y \in \pth(E)$). We shall first show that $(\nu,0)\in R$.

We may assume that $\mu \neq 0$, and write $\mu = st^{-1}$ ($s,t \in \pth(E)$). Since $\nu \neq \mu$, either $x\neq s$ or $y \neq t$. Let us assume that $x\neq s$, since the other case can be treated similarly. Also, since $R$ is an equivalence relation, $(\nu,0) \in R$ if and only if $(\mu,0) \in R$. Thus, interchanging the roles of $\mu$ and $\nu$ if necessary, we may assume that $|x| \leq |s|$. Now, $(\ra(x), x^{-1}st^{-1}y) = (x^{-1}\nu y, x^{-1}\mu y) \in R$. Since $|x| \leq |s|$ and $x \neq s$, either $x^{-1}s = 0$ or $x^{-1}s \in \pth(E)\setminus E^0$. In either case $x^{-1}st^{-1}y \neq \ra(x)$, from which it follows that $(\ra(x), 0) \in R$, by assumption. Since $R$ is a congruence, and $\nu = x\ra(x)y^{-1}$, this implies that $(\nu,0)\in R$. It follows that  $G(E)^1 \nu G(E)^1 \times \{0\} \subseteq R$, and hence $G(E)^1 \nu G(E)^1 \times G(E)^1 \nu G(E)^1 \subseteq R$, as $R$ is an equivalence relation. Letting $I \subseteq G(E)$ be the ideal generated by all $\nu \in G(E) \setminus \{0\}$ such that $(\nu, \mu) \in R$ for some $\mu \in G(E)\setminus \{\nu\}$, we conclude that $I \times I \subseteq R$. It follows that $R$ is the Rees congruence corresponding to $I$.

Conversely, suppose that $R$ is a Rees congruence, and write $$R = (I \times I) \cup \{(\mu,\mu) \mid \mu \in G(E)\},$$ where $I$ is an ideal of $G(E)$. If $v \in E^0$ is such that $(v, \mu) \in R$ for some $\mu \in G(E)\setminus \{v\}$, then $v \in I$. Hence $(v,0) \in I \times I \subseteq R$, concluding the proof of the first claim.

For the remainder of the proof, let $v \in E^0$ be such that $(v, \mu) \in R$ for some $\mu \in G(E)\setminus \{v\}$, but $(v,0) \notin R$. To prove (1), let $\mu \in S$, and write $\mu = xy^{-1}$ ($x,y \in \pth(E)$). Then $(v, vxy^{-1}v) \in R$, and since $vxy^{-1}v \neq 0$, this implies that $v = \so(x) = \so (y)$. Thus, for all $\mu, \nu \in S$ we have $(\nu, \mu\nu) = (v\nu, \mu\nu) \in R$. Since $(v, \nu) \in R$, it follows that $(v, \mu\nu) \in R$, and hence $\mu\nu \in S$, showing that $S$ is a semigroup. Furthermore, for all $\mu = xy^{-1} \in S$ we have $\mu \mu = xy^{-1}xy^{-1} \in S$, which implies that $y^{-1}x \neq 0$, and therefore either $y=xt$ or $x=yt$ for some $t \in \pth(E)$. In the first case, $\mu = xt^{-1}x^{-1}$, while in the second case, $\mu = yty^{-1}$, from which the description of the elements of $S$ in (1) follows. 

To show that $S$ is an inverse semigroup, let $\mu \in S$. Then, by the above, either $\mu = xt^{-1}x^{-1}$ or $\mu = xtx^{-1}$ for some $x,t \in \pth(E)$ with $\so(x)=v$. Let us assume that $\mu = xt^{-1}x^{-1}$, since the other case can be treated similarly. Then $(xx^{-1}, \mu) = (vxx^{-1}, \mu xx^{-1}) \in R$, and therefore $xx^{-1} \in S$. Noting that $(xtx^{-1}, xx^{-1}) = (vxtx^{-1}, \mu xtx^{-1}) \in R,$ we conclude that $\mu^{-1} = xtx^{-1} \in S$, and hence $S$ is an inverse semigroup. 

To prove (2), first note that by (1) and the assumption that $(v, \mu) \in R$ for some $\mu \in G(E)\setminus \{v\}$, the vertex $v$ cannot be a sink. Now suppose that for all $e \in E^1$ with $\so(e)=v$, there exist $f \in E^1\setminus \{e\}$ and $t \in \pth(E)$ satisfying $\so(f)=v$, $\ra(f) = \so(t)$, and $\ra(t)=\ra(e)$. We shall show that in this case $(v, 0) \in R$, contradicting our choice of $v$.

Let $\mu \in G(E)\setminus \{v\}$ be such that $(v, \mu) \in R$. By (1), either $\mu=xpx^{-1}$ or $\mu=xp^{-1}x^{-1}$ for some $x,p \in \pth(E)$ with $\so(x) =v$ and $\ra(x)=\so(p)=\ra(p)$. Let us suppose that $\mu=xpx^{-1}$, since the other case can be treated analogously. Then $xp \neq v$, since $\mu \neq v$. Write $xp=eq$ for some $e \in E^1$ and $q \in \pth(E)$. Then, by assumption there is some $f \in E^1\setminus \{e\}$ and $t \in \pth(E)$ satisfying $\so(f)=v$, $\ra(f) = \so(t)$, and $\ra(t)=\ra(e)$. Letting $s=ftq$, we have $s^{-1}xp=q^{-1}t^{-1}f^{-1}eq=0$, and therefore $$(\mu, 0) = (xpx^{-1},0) = (xps^{-1}vsx^{-1},xps^{-1}(xpx^{-1})sx^{-1}) = (xps^{-1}vsx^{-1},xps^{-1}\mu sx^{-1}) \in R.$$ Since $(v, \mu) \in R$ and $R$ is an equivalence relation, this implies that $(v,0) \in R$, as desired.
\end{proof}

Theorem 13 in~\cite{MMMP}, along with the subsequent comment, says that if $S$ is any inverse subsemigroup of $G(E)$ such that $\mu\nu \neq 0$ for all $\mu, \nu \in S$, then $S$ is generated as a semigroup by an element of the form $xpx^{-1}$ ($x,p \in \pth(E)$) and the idempotents in $S$. In particular, this applies to the inverse semigroup $S$ in Proposition~\ref{Rees-char}(1).

The next lemma shows that any vertex satisfying condition (2) in Proposition~\ref{Rees-char} produces a non-Rees congruence.

\begin{lemma} \label{non-Rees}
Let $E$ be a graph, $e \in E^1$, and $v=\so(e)$. Suppose that every $p \in \pth(E)\setminus E^0$ with $\, \so(p)=v$ and $\, \ra(p) = \ra(e)$ is of the form $p = et$ for some $t \in \pth(E)$. Then the least congruence $R \subseteq G(E) \times G(E)$ containing $\, (v, ee^{-1})$ is not a Rees congruence.
\end{lemma}

\begin{proof}
We begin by describing the elements of $P = \{(\mu v \nu, \mu ee^{-1} \nu) \mid \mu, \nu \in G(E)\}$, since $R$ is the least equivalence relation containing $P$.

For any $x,y \in \pth(E)$ with $\ra(x) = \ra(y)$, we have 
$$(xy^{-1}v, xy^{-1}ee^{-1}) = 
\left\{ \begin{array}{ll}
(x,xee^{-1}) & \text{if } y=v\\
(xy^{-1},xy^{-1}) & \text{if } y=et \text{ for some } t\in \pth(E)\\
(xy^{-1},0) & \text{if } \so(y)=v, y\neq v, \text{ and } y\neq et \text{ for all } t\in \pth(E)\\
(0,0) & \text{otherwise}.
\end{array}\right.$$
Next, let us describe products of the form $(x\nu,xee^{-1}\nu)$ belonging to $P$; i.e., ones arising from multiplying elements of the first type above on the right by $\nu \in G(E)$. For any $p,r \in \pth(E)$ with $\ra(p) = \ra(r)$, we have
$$(xpr^{-1}, xee^{-1}pr^{-1}) = 
\left\{ \begin{array}{ll}
(xr^{-1}, xee^{-1}r^{-1}) & \text{if } p=v\\
(xpr^{-1}, xpr^{-1}) & \text{if } p=et \text{ for some } t\in \pth(E)\\
(xpr^{-1},0) & \text{if } \so(p)=v, p\neq v, \text{ and } p\neq et \text{ for all } t\in \pth(E)\\
(0,0) & \text{otherwise}.
\end{array}\right.$$ We note that the elements $xr^{-1}$ and $xee^{-1}r^{-1}$, as on the first line of the previous display, are never zero, since $v = \ra(x) = \ra(r) = \so(e)$. From the computations above we see that $$P \subseteq \{(xr^{-1}, xee^{-1}r^{-1}) \mid x,r \in \pth(E),  \ra(x) = v = \ra(r)\} \cup (G(E)\times \{0\}) \cup \Delta,$$ where $\Delta = \{(\mu,\mu) \mid \mu \in G(E)\}$. To better describe $P$, we next turn to products of the form $(xy^{-1}\nu, 0)$ belonging to $P$; i.e., ones arising from multiplying elements of the third type in the first display above on the right by $\nu \in G(E)$. 

Let $I$ be the set of all elements of $G(E)$ that occur as the first coordinates of such tuples, that is $$I = \{xy^{-1}pr^{-1} \mid p,r,x,y \in \pth(E), \ \so(y)=v, \ y\neq v, \text{ and } y\neq et \text{ for all } t\in \pth(E)\},$$ and note that for any $y \in \pth(E)$ satisfying the conditions in the definition of $I$ we have $\ra (y) = y^{-1}y \in I$. We shall show that $I$ contains the ideal generated by $\ra (y)$ (provided $I\neq\emptyset$). Any nonzero element of this ideal can be expressed in the form $st^{-1}\ra(y)wz^{-1}$, for some $s,t,w,z \in \pth(E)$ satisfying $\so(t)=\ra(y)=\so(w)$. But, $$st^{-1}\ra(y)wz^{-1} = st^{-1}y^{-1}ywz^{-1} =s(yt)^{-1}(yw)z^{-1},$$ and the latter is an element of $I$, by our choice of $y$. Hence $I$ contains the ideal generated by $\ra (y)$. Since every $xy^{-1}pr^{-1} \in I$ can be expressed as $x\ra(y)y^{-1}pr^{-1}$, it further follows that $I$ is the ideal generated by all vertices $\ra(y)$, where $y \in \pth(E)$ satisfies the conditions in the definition of $I$ (if such paths exist). 

Moreover, for all $y\in \pth(E)$ of this form and all $p \in \pth(E)$ such that $\so(p) = \ra(y)$, it cannot be the case that $\ra(p) = v$, since then $s=ype$ would satisfy $\so(s)=v$ and $\ra(s) = \ra(e)$, but would not be of the form $et$ for all $t \in \pth(E)$, contrary to hypothesis. It follows that $J_{v} \not\leq_{\GJ} J_{\ra(y)}$, by Lemma~\ref{orders}(3), which implies that $v \notin I$. Similarly, for all $y$ as above and $p \in \pth(E)$ such that $\so(p) = \ra(y)$, it cannot be the case that $\ra(p) = \ra(e)$, since then $s=yp$ would satisfy $\so(s)=v$ and $\ra(s) = \ra(e)$, but not be of the form $et$ for all $t \in \pth(E)$, contrary to hypothesis. It follows that $J_{\ra(e)} \not\leq_{\GJ} J_{\ra(y)}$, and therefore $\ra(e) \notin I$.

We also observe that for any $(xpr^{-1},0) \in P$ of the form given in the third line of the description of $(xpr^{-1}, xee^{-1}pr^{-1})$ above, $xpr^{-1} \in I$, since setting $y=p$, we have $xpr^{-1} = xpy^{-1}yr^{-1}$, and $y=p$ satisfies the conditions in the definition of $I$. It follows that if $(\mu,0) \in P$ for some $\mu \in G(E) \setminus \{0\}$, then $\mu \in I$, and hence $$P \setminus \Delta = \{(xr^{-1}, xee^{-1}r^{-1}) \mid x,r \in \pth(E),  \ra(x) = v = \ra(r)\} \cup ((I\setminus \{0\}) \times \{0\}).$$ 

Now, let $$S = \{(xr^{-1}, xee^{-1}r^{-1}),(xee^{-1}r^{-1}, xr^{-1}) \mid x,r \in \pth(E),  \ra(x) = v = \ra(r)\} \cup \Delta,$$ and let $\overline{S}$ be the transitive closure of $S$. Then it is easy to see that $\overline{S}$ is an equivalence relation. We claim that $R = \overline{S} \cup (I \times I)$, from which it follows that if $(\mu, 0) \in R$ for some $\mu \in G(E) \setminus \{0\}$, then $\mu \in I$. Since, as shown above, $v \notin I$, this implies that $(v,0) \notin R$, and hence $R$ is not a Rees congruence, by Proposition~\ref{Rees-char}.

Since $P \subseteq \overline{S} \cup (I \times I) \subseteq R$, to prove that $R = \overline{S} \cup (I \times I)$, it is enough to show that $\overline{S} \cup (I \times I)$ is an equivalence relation. Since $\overline{S}$ and $I \times I$ are both equivalence relations, it suffices to show that if $(\mu,\nu) \in \overline{S} \setminus \Delta$, then $\mu \notin I$. Now, if $(\mu,\nu) \in \overline{S} \setminus \Delta$, then either $\mu = xr^{-1}$ or  $\mu = xee^{-1}r^{-1}$ for some $x,r \in \pth(E)$ with $\ra(x) = v = \ra(r)$. In the first case, if $\mu = xr^{-1} \in I$, then $v=x^{-1}(xr^{-1})r = x^{-1}\mu r$ would imply that $v \in I$, contradicting the description of $I$ above. In the second case, $\ra(e) = e^{-1}x^{-1}(xee^{-1}r^{-1})re = e^{-1}x^{-1}\mu re$ would imply that $\ra(e) \in I$, again producing a contradiction. Thus if $(\mu,\nu) \in \overline{S} \setminus \Delta$, then $\mu \notin I$, as desired.
\end{proof}

Combining the previous proposition and lemma we obtain the following generalization of a result \cite[Theorem 3.2.15]{Jones} of Jones, which deals only with graphs where every vertex is the source of some cycle and has out-degree at least $2$.

\begin{theorem} \label{only-Rees}
The following are equivalent for any graph $E$.
\begin{enumerate}
\item[$(1)$] The only congruences on $G(E)$ are Rees congruences.
\item[$(2)$] For every $e \in E^1$ there exists $p \in \pth(E) \setminus E^0$ with $\, \so(p) = \so(e)$ and $\, \ra(p)=\ra(e)$, such that $p \neq et$ for all $t \in \pth(E)$.
\end{enumerate}
\end{theorem}

\begin{proof}
If (2) holds, then $G(E)$ cannot have any non-Rees congruences, by Proposition~\ref{Rees-char}. Conversely, if (2) does not hold, then $G(E)$ has at least one non-Rees congruence, by Lemma~\ref{non-Rees}.
\end{proof}

The following easy consequence of this theorem generalizes a result \cite[Theorem 3]{AH} of Ash and Hall, which pertains only to simple graphs.

\begin{corollary} \label{cong-free}
Let $E$ be a graph such that $\, |G(E)| > 2$. Then $G(E)$ is congruence-free if and only if $E$ has only one strongly connected component, and each vertex in $E$ has out-degree at least $2$.
\end{corollary}

\begin{proof}
Suppose that $G(E)$ is congruence-free. Then the only congruences on $G(E)$ are Rees congruences. Thus $G(E)$ satisfies condition (1) of Theorem~\ref{only-Rees}, and hence also condition (2). In particular, each vertex in $E$ is either a sink or has out-degree at least $2$. Also, since every strongly connected component of $E$ corresponds to an ideal of $G(E)$, by Corollary~\ref{Green's}(3), and hence produces a congruence, there must be only one strongly connected component in $E$. This implies that either $E$ has no sinks (in which case every vertex has out-degree at least $2$), or $E$ consists of just one vertex and no edges. The latter situation is ruled out by our assumption that $|G(E)| > 2$.

Conversely, if $E$ has only one strongly connected component, then it has only one nonzero ideal, by Corollary~\ref{Green's}(3), and therefore only the Rees congruences $G(E) \times G(E)$ and $\{(\mu,\mu) \mid \mu \in G(E)\}$. If, in addition, each vertex in $E$ has out-degree at least $2$, then $E$ satisfies condition (2) of Theorem~\ref{only-Rees}, and hence no additional congruences on $G(E)$ are possible.
\end{proof}

Specializing further, we have an alternative proof of the following classical result about polycyclic monoids. (See, e.g., Section 3.4, Theorem 5 and Section 9.3, Theorem 5 in~\cite{Lawson}.) 

\begin{corollary}
The polycyclic monoid $P_n$ is congruence-free if and only if $n>1$.
\end{corollary}

\begin{proof}
As mentioned in Section~\ref{gis-sect}, $P_n$ can be viewed as the graph inverse semigroup $G(E)$, where $E$ consists of one vertex and $n$ loops. Since this graph has only one strongly connected component, the statement follows immediately from Corollary~\ref{cong-free}.
\end{proof}

By Theorem~\ref{Rees-quotients}, the quotient of a graph inverse semigroup by a Rees congruence always gives a graph inverse semigroup. However, this is not true of quotients by non-Rees congruences in general, as the next example demonstrates.

\begin{example}
Let $E$ be the following graph.
$$\xymatrix{{\bullet}^{v} \ar[r]^{e} & {\bullet}^{w}}$$
Also, let $$R = \{(v,ee^{-1}), (ee^{-1},v)\} \cup \{(\mu,\mu) \mid \mu \in G(E)\} \subseteq G(E) \times G(E).$$ Then it is easy to see that $R$ is a congruence on $G(E)$, and that $G(E)/R$ has exactly $5$ elements, three of which are idempotents (namely, $0$ and the images of $w$ and $v$). However, the only graph inverse semigroup with exactly two nonzero idempotents is the one corresponding to the graph with two vertices and no edges. Since this semigroup has three elements, it cannot be isomorphic to $G(E)/R$.
\end{example}

In contrast to the previous example, it is possible to obtain a graph inverse semigroup as the quotient of another such semigroup by a non-Rees congruence, as the next example shows.

\begin{example} \label{con-eg}
Let $E$ be the following graph.
$$
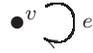
\xymatrix{{\bullet}^{v} \ar@(ur,dr)^{e}}$$
Also, let $R \subseteq G(E) \times G(E)$ be the least congruence containing $(v,e)$. Since $G(E) \setminus \{0\}$ is a semigroup (the bicyclic semigroup, as usually defined), $(0,\mu) \in R$ only if $\mu = 0$. Therefore, $R$ is not a Rees congruence, by Proposition~\ref{Rees-char}. Now, $(e^{-1}, v) = (e^{-1}v,e^{-1}e) \in R$, from which it is easy to see that $(\mu, \nu) \in R$ for all $\mu, \nu \in G(E) \setminus \{0\}$. It follows that $G(E)/R \cong G(F)$, where $F$ is a graph having only one vertex and no edges.
\end{example}

\section{Idempotents} \label{idempt-sect}

An element $\mu$ of a semigroup is an \emph{idempotent} if $\mu\mu = \mu$. In this section we recall some basic facts about idempotents in inverse semigroups and record some observations about the idempotents of $G(E)$ that will be useful throughout the rest of the paper. All of the results about $G(E)$ are easy, and most have been previously observed elsewhere (e.g., \cite{Jones, JL}), but we give the proofs here for completeness.

Given an inverse semigroup $S$, the \emph{natural partial order} $\, \leq$ on $S$ is defined by $\mu \leq \nu$ ($\mu, \nu \in S$) if $\mu=\epsilon\nu$ for some idempotent $\epsilon \in S$. (See \cite[Section 5.2]{Howie} for details.) Furthermore restricting $\leq$ to the subset $I$ of $S$ consisting of all the idempotents makes $(I,\leq)$ a \emph{lower semilattice} \cite[Proposition 1.3.2]{Howie}, that is, a partially ordered set where every pair of elements has a greatest lower bound.

\begin{lemma} \label{nat-ord}
Let $E$ be a graph, let $\, \leq$ be the natural partial order on $G(E)$, and let $I$ be the subset of idempotents of $G(E)$. Then the following hold.
\begin{enumerate}
\item[$(1)$] An element $\mu \in G(E) \setminus \{0\}$ is in $I$ if and only if $\mu = xx^{-1}$ for some $x \in \pth(E)$.
\item[$(2)$] Let $u,v,x,y \in \pth (E)$ be such that $\, \ra(u)=\ra(v)$ and $\, \ra(x)=\ra(y)$. Then $uv^{-1} \leq xy^{-1}$ if and only if $u=xt$ and $v=yt$ for some $t \in \pth(E)$.
\item[$(3)$] An idempotent $\mu \in G(E)$ is maximal in $I$ with respect to $\, \leq$ if and only if $\mu \in E^0$.
\item[$(4)$] An idempotent $\mu \in G(E)$ is maximal in $I\setminus E^0$ with respect to $\, \leq$ if and only if $\mu =ee^{-1}$ for some $e\in E^1$.
\end{enumerate}
\end{lemma}

\begin{proof}
(1) If $S$ is any inverse semigroup and $\mu \in S$ is an idempotent, then $\mu\mu\mu = \mu$, and hence $\mu = \mu^{-1}$. Applying this to $G(E)$, suppose that $xy^{-1} \in G(E)$ is an idempotent ($x, y \in \pth(E)$). Then $xy^{-1} = (xy^{-1})^{-1} =yx^{-1}$, from which the desired statement follows.

\vspace{\baselineskip}

(2) Suppose that $u=xt$ and $v=yt$ for some $t \in \pth(E)$. Then $$uv^{-1} = xtt^{-1}y^{-1} = (xtt^{-1}x^{-1}) xy^{-1},$$ which implies that $uv^{-1} \leq xy^{-1}$, since $xtt^{-1}x^{-1}$ is an idempotent. 

For the converse, suppose that $uv^{-1} \leq xy^{-1}$. Then $uv^{-1} = (pp^{-1}) xy^{-1}$ for some $p \in \pth(E)$, by (1). Since $pp^{-1} xy^{-1} \neq 0$, there is some $t \in \pth(E)$ such that either $x=pt$ or $p = xt$. In the first case, $uv^{-1} = pp^{-1} xy^{-1} =  xy^{-1}$, and hence $u=xt$ and $v=yt$, where $t=\ra(x)=\ra(y)$. In the second case, $uv^{-1} = pp^{-1} xy^{-1} = xtt^{-1}y^{-1}$, and hence $u=xt$ and $v=yt$, as desired. 

\vspace{\baselineskip}

(3) Suppose that $\mu \in E^0$ and $\mu \leq \nu$ for some $\nu \in I$. Then $\nu \neq 0$, and hence, by (1) and (2), $\nu = xx^{-1}$ and $\mu=xtt^{-1}x^{-1}$ for some $x,t \in \pth(E)$. Since $\mu$ is a vertex, this can happen only if $\nu = x = t = \mu$, and hence $\mu$ is maximal. 

Conversely, suppose that $\mu \in G(E)$ is an idempotent maximal in $I$. Then $\mu \neq 0$, and hence $\mu = xx^{-1}$ for some $x \in \pth(E)$, by (1). Thus $\mu = xx^{-1} \leq \so(x)$, by (2). Since $\mu$ is maximal, this implies that $\mu = x = \so(x)$, and hence $\mu \in E^0$.

\vspace{\baselineskip}

(4) Let $e \in E^1$, and suppose that $ee^{-1} \leq \nu$ for some $\nu \in I \setminus E^0$. Then $\nu \neq 0$, and hence, by (1) and (2), $\nu = xx^{-1}$ and $ee^{-1} = xtt^{-1}x^{-1}$ for some $x,t \in \pth(E)$. Since $e \in E^1$, this implies that either $e = x$ and $t=\ra(e)$, or $e = t$ and $x=\so(e)$. In the second case, $\nu \in E^0$, contrary to assumption. Thus $e=x$, and therefore $\nu = ee^{-1}$. Hence $ee^{-1}$ is maximal in $I\setminus E^0$.

Conversely, suppose that $\mu \in G(E)$ is an idempotent maximal in $I\setminus E^0$. Then $\mu \neq 0$, and hence $\mu = xx^{-1}$ for some $x \in \pth(E)$, by (1). Since $xx^{-1} \notin E^0$, we can write $x=et$ for some $e \in E^1$ and $t \in \pth(E)$, and hence $\mu = ett^{-1}e^{-1} \leq ee^{-1}$, by (2). Since $\mu$ is maximal in $I\setminus E^0$, and $ee^{-1} \in I\setminus E^0$, this implies that $\mu = ee^{-1}$ (i.e., $t \in E^0$).
\end{proof}

Given an inverse semigroup $S$, the following relation is called the \emph{maximum idempotent-separating congruence} on $S$: $$\{(\mu, \nu) \mid \mu,\nu \in S \text{ and }\mu^{-1}\epsilon\mu = \nu^{-1}\epsilon \nu \text{ for all idempotents } \epsilon \in S\}.$$ The semigroup $S$ is \emph{fundamental} if this relation is equal to the diagonal congruence.

\begin{lemma} \label{fund}
The inverse semigroup $G(E)$ is fundamental for any graph $E$.
\end{lemma}

\begin{proof}
It is a standard fact \cite[Proposition 5.3.7]{Howie} that in an inverse semigroup the maximum idempotent-separating congruence is the largest congruence contained in $\GH$. Now, by Corollary~\ref{Green's}(4), $\mu \, \GH \, \nu$ if and only if $\mu=\nu$, for all $\mu, \nu \in G(E)$. Thus in a graph inverse semigroup $\GH$ is precisely the diagonal congruence, and therefore so is the maximum idempotent-separating congruence, showing that $G(E)$ is fundamental.
\end{proof}

\section{Representations}\label{rep-sect}

Recall that given a nonempty set $X$, a binary relation $R \subseteq X \times X$ is a \emph{partial function} if $(x,y), (x,z) \in R$ implies that $y=z$ for all $x, y, z \in X$. It is a standard fact that the set $\GP_X$  of all partial functions on $X$ is a semigroup, under composition of relations \cite[Proposition 1.4.2]{Howie}, called the \emph{partial transformation semigroup} on $X$. Given a semigroup $S$, a semigroup homomorphism $\phi : S \to \GP_X$ is a called a \emph{representation of} $S$ \emph{by partial transformations}. If $\phi$ is injective, then it is a \emph{faithful} representation. The cardinality of $X$ is called the \emph{degree} of $\phi$.

Our next goal is to find the minimum possible degree of a faithful representation of $G(E)$ by partial transformations, when $G(E)$ is countable. If $G(E)$ is countably infinite, then it does not have a faithful representation by
partial transformations on any finite set (since there are only finitely many
such partial transformations). Hence, in this case, the minimum possible degree
of a faithful partial transformation representation of $G(E)$ is $|G(E)| = \aleph_0$. The usual Vagner-Preston representation of $G(E)$ (see \cite[Theorem 5.1.7]{Howie}) is an example of such a representation with minimum degree. 

Turning to finite graph inverse semigroups, we note that $G(E)$ is finite precisely when $E$ is finite and acyclic. To determine the minimum possible degree of a faithful representation of $G(E)$ by partial transformations we shall need the following theorem of Easdown. Before stating the result, we recall that an element $x$ of a partially ordered set $X$ is called \emph{join-irreducible} if it is not zero (i.e., the least element of $X$, when it exists), and $x=y\vee z$ implies that $x=y$ or $x=z$, for all $y,z\in X$ (where $y\vee z$ denotes the least upper bound of $y$ and $z$, if it exists).

\begin{theorem}[Theorem 7 in~\cite{Easdown}] \label{Easdown-thrm}
Let $S$ be a finite fundamental inverse semigroup. Then the minimum possible degree of a faithful representation of $S$ by partial transformations equals the number of join-irreducible idempotents in $S$.
\end{theorem}

We note that Easdown's proof of this theorem gives an explicit construction of a faithful representation having the minimum possible degree.

Next, let us describe the join-irreducible idempotents of $G(E)$.

\begin{lemma} \label{join-irred}
Let $E$ be a graph, let $\, \leq$ be the natural partial order on $G(E)$, and let $x\in \pth(E)$. Then the idempotent $xx^{-1}$ is join-irreducible in the lower semilattice of idempotents of $G(E)$ if and only if the out-degree of $\, \ra(x)$ is at most $\, 1$.
\end{lemma}

\begin{proof}
Suppose that the out-degree of $\ra(x)$ is at least $2$. Then there are $e,f\in E^1$ such that $e\not=f$ and $\so(e)=\so(f)=\ra(x)$. Hence $xx^{-1}=xee^{-1}x^{-1}\vee xff^{-1}x^{-1}$, by Lemma~\ref{nat-ord}(2), and so $xx^{-1}$ is not join-irreducible. 

For the converse, suppose that the out-degree of $\ra(x)$ is at most $1$. If the out-degree of $\ra(x)$ is $0$, then, by Lemma~\ref{nat-ord}(2), the only idempotent $\tau$ such that $\tau < xx^{-1}$ is $\tau = 0$. Therefore $xx^{-1}$ is clearly join-irreducible in this case. Let us therefore assume that out-degree of $\ra(x)$ is $1$, and that $xx^{-1}=\mu\vee\nu$ for some idempotents $\mu,\nu\in G(E)$. If $\mu=0$ or $\nu=0$, then $xx^{-1}=\nu$ or $xx^{-1}=\mu$, respectively. Hence we may also assume that $\mu=yy^{-1}$ and $\nu=zz^{-1}$ for some distinct $y,z\in \pth(E)$, where, without loss of generality, $\mu \neq xx^{-1}$. Then, by Lemma~\ref{nat-ord}(2), $y=xeu$ for some $u\in \pth(E)$, where $e \in E^1$ is the unique edge satisfying $\so(e) = \ra(x)$. If $z\not=x$, then, similarly, $z=xev$ for some $v\in \pth(E)$. But then $yy^{-1}\vee zz^{-1}=xee^{-1}x^{-1} \neq xx^{-1}$, contradicting $xx^{-1}=\mu\vee\nu$. Thus $z=x$, and so $xx^{-1}$ is join-irreducible. 
\end{proof}

\begin{proposition} \label{repres}
Let $G(E)$ be a finite graph inverse semigroup. Then the minimum possible degree of a faithful representation of $G(E)$ by partial transformations is the number of paths $x \in \pth(E)$ such that the out-degree of $\, \ra(x)$ is at most $\, 1$.
\end{proposition}

\begin{proof}
Since, by Lemma~\ref{fund}, $G(E)$ is fundamental, we can apply Theorem~\ref{Easdown-thrm} to it. The proposition now follows from Lemma~\ref{join-irred}, since, by Lemma~\ref{nat-ord}(1), all nonzero idempotents of $G(E)$ are of the form $xx^{-1}$, for some $x \in \pth(E)$.
\end{proof}

\section{Homomorphisms}

Next, we describe when a homomorphism of graphs can be extended to a homomorphism of the corresponding graph inverse semigroups. 

\begin{theorem} \label{graph-iso}
Let $E_a$ and $E_b$ be two graphs, and suppose that $\phi_0 : E_a^0 \to E_b^0$ and $\phi_1 : E_a^1 \to E_b^1$ are functions such that $\phi = (\phi_0, \phi_1)$ is a graph homomorphism from $E_a$ to $E_b$. Then the following are equivalent:
\begin{enumerate}
\item[$(1)$] $\phi$ can be extended to a semigroup homomorphism $\varphi : G(E_a) \to G(E_b)$ that takes zero to zero,
\item[$(2)$] $\phi_0$ and $\phi_1$ are injective.
\end{enumerate}
If these conditions hold, then $\varphi$ is uniquely determined and injective. Moreover, $\varphi$ is surjective if and only if $\phi_0$ and $\phi_1$ are surjective.
\end{theorem}

\begin{proof}
Suppose that (1) holds. If $\phi_0$ is not injective, then there exist distinct $v, w \in E_a^0$ such that $\phi_0(v)=\phi_0(w)$. Hence $$0 = \varphi(0) = \varphi(vw) = \varphi(v)\varphi(w) = \phi_0(v)\phi_0(w) = \phi_0(v)\phi_0(v),$$ which is impossible, since $\phi_0(v) \in E_b^0$. Thus $\phi_0$ must be injective.

If $\phi_1$ is not injective, then there exist distinct $e,f \in E_a^1$ such that $\phi_1(e)=\phi_1(f)$. Since $\varphi$ is a homomorphism of inverse semigroups, $\varphi(\mu^{-1}) = \varphi(\mu)^{-1}$ for all $\mu \in G(E_a)$, and hence $$0 = \varphi(0) =  \varphi(f^{-1}e) = \varphi(f)^{-1}\varphi(e) = \phi_1(f)^{-1}\phi_1(e) = \phi_1(e)^{-1}\phi_1(e).$$ This is impossible, since $\phi_1(e) \in E_b^1$. Thus $\phi_1$ must be injective, showing that (2) holds.

Conversely, suppose that (2) holds. As noted in Section~\ref{graphs-sect}, any nonzero element $\mu \in G(E_a)$ can be written uniquely in the form $\mu = ve_1\dots e_n f_m^{-1}\dots f_1^{-1}w$ for some $v, w \in E_a^0$, $e_1, \dots, e_n, f_1, \dots, f_m \in E_a^1$, and $m,n \in \N$ (with $n=0$ signifying that the ``path" part of $\mu$ is just the vertex $v$, and analogously for $m$). Thus we can define $\varphi : G(E_a) \to G(E_b)$ by $$\varphi (ve_1\dots e_n f_m^{-1}\dots f_1^{-1}w) = \phi_0(v)\phi_1(e_1)\dots \phi_1(e_n) \phi_1(f_m)^{-1}\dots \phi_1(f_1)^{-1}\phi_0(w),$$ and $\varphi(0) = 0$.

To show that $\varphi$ is a semigroup homomorphism, let $\mu, \nu \in G(E_a)$. If either $\mu = 0$ or $\nu = 0$, then clearly $\varphi(\mu)\varphi(\nu) = 0 = \varphi(\mu\nu)$. Let us therefore assume that $\mu \neq 0$ and $\nu \neq 0$.

Suppose that $\mu\nu = 0$, and write $\mu = sf_m^{-1}\dots f_1^{-1}v$, $\nu = we_1\dots e_ny^{-1}$ ($s,y \in \pth(E_a)$, $v, w \in E_a^0$, $e_1, \dots, e_n, f_1, \dots, f_m \in E_a^1$, and $m,n \in \N$). Then either $v \neq w$, or $v=w$, $e_1 = f_1, \dots, e_{l-1}=f_{l-1}$, but $e_l \neq f_l$ for some $l \geq 1$. In the first case, $\phi_0(v)\phi_0(w) = 0$, by the injectivity of $\phi_0$, and hence $\varphi(\mu\nu) = 0 = \varphi(\mu)\varphi(\nu)$. In the second case $$\varphi(\mu)\varphi(\nu) = \phi_0(\so(s))\dots \phi_1(f_m)^{-1}\dots \phi_1(f_l)^{-1}\phi_1(e_l)\dots \phi_1(e_n) \dots \phi_0(\so(y)) = 0,$$ since $\phi_1(f_{l}) \neq \phi_1(e_{l})$, by the injectivity of $\phi_1$. Thus $\varphi(\mu)\varphi(\nu) = 0 = \varphi(\mu\nu)$. 

Let us therefore assume that $\mu\nu \neq 0$, and write $\mu = st^{-1}$ and $\nu = xy^{-1}$ ($s,t,x,y \in \pth(E_a)$). Then there is some $p \in \pth(E_a)$ such that either $t=xp$ or $x=tp$. Let us assume that $t=xp$, since the other case is similar. Then, using the definition of $\varphi$ and the fact that $\phi$ is a graph homomorphism, we see that $$\varphi (st^{-1}) = \varphi (s)\varphi (xp)^{-1} = \varphi (s) (\varphi (x) \varphi (p))^{-1} = \varphi (s) \varphi (p)^{-1} \varphi (x)^{-1}.$$ Hence $$\varphi(\mu)\varphi(\nu) = \varphi(s)\varphi(p)^{-1}\varphi(x)^{-1}\varphi(x)\varphi(y)^{-1} = \varphi(s)\varphi(p)^{-1}\varphi(y)^{-1} = \varphi(sp^{-1}y^{-1}) = \varphi(\mu\nu),$$ showing that $\varphi$ is a homomorphism, whose restrictions to $E_a^0$ and $E_a^1$ are $\phi_0$ and $\phi_1$, respectively. That is, (1) holds. 

Next, we note that $\varphi$ is uniquely determined, since $E^0_a \cup E^1_a \cup \{0\}$ is a generating set for $G(E_a)$ as an inverse semigroup, and hence the value of any homomorphism to another inverse semigroup is determined by its values on this set. Also, it follows immediately from the definition of $\varphi$ and the injectivity of $\phi_0$ and $\phi_1$ that $\varphi$ is injective. The final claim follows from the fact that the inverse subsemigroup of $G(E_b)$ generated by $\phi_0 (E^0_a) \cup \phi_1(E^1_a) \cup \{0\}$ is $\varphi (G(E_a))$.
\end{proof}

The next example shows that in the previous theorem it is necessary to assume that $\varphi$ preserves zero for (1) to be equivalent to (2).

\begin{example}
Consider the following two graphs.
$$E_a = \xymatrix{{\bullet}^{v_1} \ \ {\bullet}^{v_2}} \ \ \ \ \ \ \ \ \ \ \ \ \ \ E_b = \xymatrix{{\bullet}^{w}}$$
Define $\phi_0 : E_a^0 \to E_b^0$ by $\phi_0(v_1) = w = \phi_0(v_2)$, and let $\phi_1 : E_a^1 \to E_b^1$ be the empty function. Then $\phi = (\phi_0,\phi_1)$ defines a graph homomorphism from $E_a $ to $E_b$, where $\phi_0$ is clearly not injective. However, $\phi$ can be extended to the semigroup homomorphism $\varphi: G(E_a) \to G(E_b)$ that takes all elements of $G(E_a)$ (including $0$) to $w$.
\end{example}

To complement Theorem~\ref{graph-iso}, next we show that an isomorphism of graph inverse semigroups always restricts to an isomorphism of the underlying graphs. In the case where the graphs are finite, this follows from a result \cite[Corollary 3.2]{Krieger} of Krieger.

\begin{proposition} \label{isomorphisms}
Let $E_a$ and $E_b$ be two graphs, and let $\varphi : G(E_a) \to G(E_b)$ be a semigroup isomorphism. Then letting $\phi_0$ and $\phi_1$ be the restrictions of $\varphi$ to $E_a^0$ and $E_a^1$, respectively, gives a graph isomorphism $\phi = (\phi_0, \phi_1)$ from $E_a$ to $E_b$.
\end{proposition}

\begin{proof}
Let $I_a \subseteq G(E_a)$ and $I_b \subseteq G(E_b)$ denote the respective subsets of idempotents. Also let $\leq_a^I$ and $ \leq_b^I$ denote the restrictions to $I_a$ and $I_b$, respectively, of the natural partial orders on $ G(E_a)$ and $G(E_b)$, respectively. (As mentioned in Section~\ref{idempt-sect}, $(I_a, \leq_a^I)$ and $(I_b, \leq_b^I)$ are lower semilattices.) By Lemma~\ref{nat-ord}(3), every vertex in $E_a^0$, but no other element of $G(E_a)$, is maximal in $I_a$ with respect to $\leq_a^I$, and analogously for $G(E_b)$. Since any isomorphism of semigroups induces an order-isomorphism of the corresponding idempotent semilattices, $\varphi$ must take $E_a^0$ bijectively to $E_b^0$.

Next, by Lemma~\ref{nat-ord}(4), every element of the form $ee^{-1}$ ($e \in E_a^1$), but no other element of $G(E_a)$, is maximal in $I_a \setminus E_a^0$ with respect to $\leq_a^I$, and analogously for $G(E_b)$. Hence $\varphi$ must take $\{ee^{-1} \mid e \in E_a^1\}$ bijectively to $\{ff^{-1} \mid f \in E_b^1\}$. Now, let $e \in E_a^1$ be any edge, write $\varphi(ee^{-1}) = ff^{-1}$ for some $f \in E_b^1$, and write $\varphi(e) = xy^{-1}$ for some $x,y \in \pth(E_b)$. Then $$ff^{-1} = \varphi(ee^{-1}) = \varphi(e)\varphi(e)^{-1} = xy^{-1}yx^{-1} = xx^{-1},$$ since $\varphi$ is an isomorphism of inverse semigroups. It follows that $x = f$. Furthermore, $$\varphi(\ra_a(e)) = \varphi(e^{-1}e) = yf^{-1}fy^{-1} = yy^{-1},$$ which implies that $y \in E_b^0$, since $\varphi (E_a^0) = E_b^0$. Therefore $\varphi(e) = f$, and hence $\varphi (E_a^1) \subseteq E_b^1$. Since $\varphi$ takes $\{ee^{-1} \mid e \in E_a^1\}$ bijectively to $\{ff^{-1} \mid f \in E_b^1\}$, it follows that $\varphi$ takes $E_a^1$ bijectively to $E_b^1$. Moreover, since $$0 \neq \varphi(e) = \varphi(\so_a(e)e\ra_a(e)) = \varphi(\so_a(e))\varphi(e)\varphi(\ra_a(e))$$ for any $e \in E_a^1$, we conclude that $\varphi(\so_a(e)) = \so_b(\varphi(e))$ and $\varphi(\ra_a(e)) = \ra_b(\varphi(e))$. Therefore, letting $\phi_0$ and $\phi_1$ be the restrictions of $\varphi$ to $E_a^0$ and $E_a^1$, respectively, gives a graph isomorphism $\phi = (\phi_0, \phi_1)$ from $E_a$ to $E_b$.
\end{proof}

While, by the above result, any isomorphism of graph inverse semigroups induces an isomorphism of the corresponding graphs, it is not the case in general that a homomorphism of graph inverse semigroups induces a homomorphism of the corresponding graphs, even when the homomorphism is injective or surjective, as the next two examples show. 

\begin{example}
Consider the following two graphs.
$$E_a = \xymatrix{{\bullet}^{w}} \ \ \ \ \ \ \ \ \ \ \ \ \ \ E_b = \xymatrix{{\bullet}^{v_1} \ar [r] ^{e} & {\bullet}^{v_2}}$$
Then $G(E_a) = \{0,w\}$, and thus it is easy to see that $\varphi(0) = 0$, $\varphi(w) = ee^{-1}$ defines an injective semigroup homomorphism $\varphi : G(E_a) \to G(E_b)$. However, the restriction of $\varphi$ to $E_a^0 = \{w\}$ is not a function $E_a^0 \to E_b^0$, and in particular, $\varphi$ does not induce a graph homomorphism from $E_a$ to $E_b$.
\end{example}

\begin{example}
Consider the following two graphs.
$$E_a = \xymatrix{{\bullet}^{v}\ar@(ur,dr)^{e}} \ \ \ \ \ \ \ \ \ \ \ \ \ \ E_b = \xymatrix{{\bullet}^{w}}$$
Then it is easy to see that $$\varphi(\mu) = 
\left\{ \begin{array}{ll}
w & \text{if } \mu \neq 0\\
0 & \text{if } \mu = 0
\end{array}\right.$$ 
defines a surjective semigroup homomorphism $\varphi : G(E_a) \to G(E_b)$ (cf.\ Example~\ref{con-eg}). However, the restriction of $\varphi$ to $E_a^1 = \{e\}$ is not a function $E_a^0 \to E_b^0 = \emptyset$, and in particular, $\varphi$ does not induce a graph homomorphism from $E_a$ to $E_b$.
\end{example}

From Theorem~\ref{graph-iso} and Proposition~\ref{isomorphisms} we immediately obtain the following well-known result. (In the case of simple graphs it is noted by Ash and Hall in~\cite{AH} after Theorem 1, and in the case of finite graphs it is proved by Krieger in \cite[Corollary 3.2]{Krieger}. It also follows from the result \cite[Corollary 8.5]{CS} of Costa and Steinberg that two graph inverse semigroups are Morita equivalent if and only if the underlying graphs are isomorphic.)

\begin{corollary} \label{cor-iso}
Let $E_a$ and $E_b$ be two graphs. Then $E_a \cong E_b$ if and only if $G(E_a) \cong G(E_b)$. 
\end{corollary}

We conclude with several other consequences of Theorem~\ref{graph-iso} and Proposition~\ref{isomorphisms}.

\begin{corollary} \label{cor-autos}
Let $E$ be a graph. Denote by $\, \Aut(E)$ and $\, \Aut(G(E))$ the groups of automorphisms of $E$ as a graph and $G(E)$ as a semigroup, respectively. Then $\, \Aut(G(E)) \cong \Aut(E)$ as groups.
\end{corollary}

\begin{proof}
Let $\varphi \in \Aut(G(E))$ be any automorphism. Then, by Proposition~\ref{isomorphisms}, letting $\varphi_0$ and $\varphi_1$ be the restrictions of $\varphi$ to $E^0$ and $E^1$, respectively, gives a graph automorphism $(\varphi_0, \varphi_1)$ of $E$. Hence we can define a function $\psi : \Aut(G(E)) \to \Aut(E)$ by $\psi(\varphi) = (\varphi_0, \varphi_1)$. Moreover, by Theorem~\ref{graph-iso}, $\psi$ is a bijection. 

Now, if $\varphi, \varphi' \in \Aut(G(E))$ are two automorphisms, then, again by Proposition~\ref{isomorphisms}, the restrictions of $\varphi \circ \varphi'$ to $E^0$ and $E^1$ are precisely $\varphi_0 \circ \varphi'_0$ and $\varphi_1 \circ \varphi'_1$, respectively. It follows that $\psi : \Aut(G(E)) \to \Aut(E)$ is a group isomorphism.
\end{proof}

\begin{corollary} \label{group-cor}
For every group $H$ there is some graph $E$ such that $H \cong \Aut(G(E))$.
\end{corollary}

\begin{proof}
By Frucht's theorem~\cite{Frucht}, every group is isomorphic to the automorphism group of some graph. The claim now follows by combining this fact with Corollary~\ref{cor-autos}.
\end{proof}

\begin{corollary} \label{poset-cor}
Let $E$ be a simple acyclic graph, let $J_{G(E)}$ be the set of nonzero $\GJ$-classes of $G(E)$, and let $\, \Aut(J_{G(E)}, \leq_{\GJ})$ denote the group of order-automorphisms of $\, (J_{G(E)}, \leq_{\GJ})$. Then $\, \Aut(J_{G(E)}, \leq_{\GJ}) \cong \Aut(G(E))$.
\end{corollary}

\begin{proof}
Since $E$ is acyclic, as noted immediately after Corollary~\ref{Green's}, the elements of $J_{G(E)}$ are in one-to-one correspondence with the vertices of $E$. Moreover, for all $u,v \in E^0$, by Lemma~\ref{orders}(3), $J_{u}\leq_{\GJ} J_{v}$ if and only if $\so(t)=v$ and $\ra(t)=u$ for some $t \in \pth(E)$. It is now easy to see that every automorphism of $E$ induces an order-automorphism of $J_{G(E)}$, and vice versa. It follows that $\Aut(E) \cong \Aut(J_{G(E)}, \leq_{\GJ})$, from which we obtain the result, by Corollary~\ref{cor-autos}.
\end{proof}




\vspace{.1in}

\noindent Z.\ Mesyan, Department of Mathematics, University of Colorado, Colorado Springs, CO 80918, USA 

\noindent \emph{Email:} \href{mailto:zmesyan@uccs.edu}{zmesyan@uccs.edu}  \newline

\noindent J.\ D.\ Mitchell, Mathematical Institute, North Haugh, St Andrews, Fife, KY16 9SS, Scotland

\noindent \emph{Email:} \href{mailto:jdm3@st-and.ac.uk}{jdm3@st-and.ac.uk} \newline

\end{document}